%% file: root.tex
\documentclass[journal]{IEEEtran}

\usepackage{amsmath}
\usepackage{amsthm}
\usepackage{amssymb}

\newtheorem{thm}{Theorem}
\newtheorem{lem}{Lemma}

\newtheorem{defn}{Definition}
\newtheorem{rem}{Remark}

\begin{document}
%
\title{Necessary and Sufficient Conditions for\\ Difference Flatness}
%
%
%

\author{Bernd Kolar, Johannes Diwold, and Markus Sch\"{o}berl%
\thanks{The first author and the second author have been supported by the Austrian Science Fund (FWF) under grant number P~29964 and P~32151.}
\thanks{All authors are with the Institute of Automatic Control and Control Systems Technology, Johannes Kepler University Linz, Altenbergerstra\ss e 69, 4040 Linz, Austria (e-mail: bernd\underline{\ }kolar@ifac-mail.org, johannes.diwold@jku.at, markus.schoeberl@jku.at). }}

\maketitle

\begin{abstract}
We show that the flatness of a nonlinear discrete-time system can be checked by computing a unique sequence of involutive distributions. 
The well-known test for static feedback linearizability is included as a special case.
Since the computation of the sequence of distributions requires only the solution of algebraic equations, it allows an efficient implementation in a computer algebra program.
In case of a positive result, a flat output can be obtained by straightening out the involutive distributions with the Frobenius theorem.

\textit{\copyright \ 2022 IEEE.  Personal use of this material is permitted.  Permission from IEEE must be obtained for all other uses, in any current or future media, including reprinting/republishing this material for advertising or promotional purposes, creating new collective works, for resale or redistribution to servers or lists, or reuse of any copyrighted component of this work in other works.\\ DOI 10.1109/TAC.2022.3151615}
\end{abstract}


%
\IEEEpeerreviewmaketitle

\input{IEEE_TAC_TN_sub}

\bibliographystyle{IEEEtran}
\bibliography{IEEEabrv,Bibliography_Bernd}

\end{document}

%% file: IEEE_TAC_TN_sub.tex
\section{Introduction}

The concept of flatness has been introduced by Fliess, L{\'e}vine,
Martin and Rouchon in the 1990s for nonlinear continuous-time systems
(\cite{FliessLevineMartinRouchon:1992}, \cite{FliessLevineMartinRouchon:1995},
\cite{FliessLevineMartinRouchon:1999}). For discrete-time systems,
flatness can be defined analogously to the continuous-time case in
a straightforward way by replacing time derivatives with forward-shifts
(see e.g. \cite{Sira-RamirezAgrawal:2004}, \cite{KaldmaeKotta:2013},
or \cite{KolarKaldmaeSchoberlKottaSchlacher:2016}). To distinguish
both concepts, the terms differential flatness and difference flatness
are commonly used. With this definition, like in the continuous-time
case, flatness is equivalent to linearizability by an endogenous dynamic
feedback as it is defined for discrete-time systems in \cite{Aranda-BricaireMoog:2008}.
We do not consider backward-shifts like in \cite{GuillotMillerioux:2020}
or \cite{MilleriouxJungers:2021}. According to
the terminology of \cite{Aranda-BricaireMoog:2008}, such systems
would be linearizable by an exogenous dynamic feedback. The general
dynamic feedback linearization problem and related topics like partial
feedback linearization are studied e.g. in \cite{Aranda-BricaireKottaMoog:1996},
\cite{CalifanoMonacoNormand:1999}, or \cite{CalifanoMonacoNormand:2002}.

The static feedback linearization problem is a special case of the
endogenous dynamic feedback linearization problem and has been solved
for discrete-time systems in \cite{Grizzle:1986}, \cite{Jakubczyk:1987},
and \cite{Aranda-BricaireKottaMoog:1996}. The approach of \cite{Grizzle:1986}
is similar to the well-known approach of \cite{JakubczykRespondek:1980}
and \cite{vanderSchaft:1984} for continuous-time systems, and allows
to check whether a discrete-time system is static feedback linearizable
or not by computing a certain sequence of involutive distributions.
This test requires only the solution of algebraic equations. Subsequently,
a linearizing output can be obtained by straightening out these distributions
with the Frobenius theorem, which requires the solution of ODEs.

For checking the flatness of nonlinear discrete-time systems, until
now no comparable, computationally similarly efficient test is available.
In \cite{Sato:2012} a sufficient condition for flatness is given,
and \cite{KaldmaeKotta:2013} formulates necessary and sufficient
conditions that are analogous to those of \cite{Levine:2011} for
continuous-time systems and rather hard to check. More recently, in
\cite{KolarSchoberlDiwold:2019} we have shown that every flat discrete-time
system can be decomposed by coordinate transformations into a subsystem
and an endogenous dynamic feedback. The importance of this observation
lies in the fact that the complete system is flat if and only if the
subsystem is flat. Thus, a repeated application of the decomposition
allows to check whether a system with an $n$-dimensional state space
is flat or not in at most $n-1$ steps and yields a flat output. However,
since the transformations that achieve the decompositions are constructed
by straightening out vector fields / distributions with the flow-box
theorem / Frobenius theorem, this procedure is still not as computationally
efficient as the solution of the static feedback linearization problem
in \cite{Grizzle:1986}. Therefore, the present paper proposes an
a priori test which only checks whether the repeated decompositions
are possible or not, without actually performing them. A similar problem
is addressed in \cite{Schlacher:2019}, but with a completely different
approach that is based on certain normalized system representations
and exterior algebra. Here, in contrast, we introduce a straightforward
generalization of the sequence of distributions from the static feedback
linearization test of \cite{Grizzle:1986}. The difference to the
sequence of \cite{Grizzle:1986} is that in every step of its construction
we do not proceed with the previously constructed distribution itself
but rather with its largest ``projectable'' subdistribution. Like
in \cite{Grizzle:1986} the proposed sequence of distributions is
unique, allows an efficient computation, and the system is flat if
and only if the dimension of the last distribution is equal to the
dimension of the state space. Thus, we separate the problem of checking
flatness from the problem of finding a flat output in the same way
as it is possible for static feedback linearizability. If the test
yields a positive result, a flat output can be obtained by straightening
the distributions out with the Frobenius theorem.

The paper is organized as follows: In Section \ref{sec:DiscreteTimeSystems_and_Flatness}
and Section \ref{sec:Projectable_Distributions} we recapitulate the
concept of flatness for discrete-time systems and provide some background
on projectable vector fields and distributions, which is the mathematical
foundation for our conditions. In particular, we prove that the largest
projectable subdistribution of a given distribution is unique and
show how it can be computed. In Section \ref{sec:Necessary_and_Sufficient}
we present our main result: We introduce a sequence of distributions
which generalizes the sequence of distributions from the static feedback
linearization test, and show that it gives rise to necessary and sufficient
conditions for flatness. An example in Section \ref{sec:Example}
illustrates our results.

\section{\label{sec:DiscreteTimeSystems_and_Flatness}Discrete-Time Systems
and Flatness}

In this contribution we consider discrete-time systems
\begin{equation}
x^{i,+}=f^{i}(x,u)\,,\quad i=1,\ldots,n\label{eq:sys}
\end{equation}
in state representation with $\dim(x)=n$, $\dim(u)=m$, and smooth
functions $f^{i}(x,u)$ that satisfy the submersivity condition
\begin{equation}
\mathrm{rank}(\partial_{(x,u)}f)=n\,.\label{eq:submersivity}
\end{equation}
Submersivity is a usual assumption in the discrete-time literature,
and since it is necessary for accessibility (see e.g. \cite{Grizzle:1993}),
it is actually no restriction.\footnote{In \cite{GuillotMillerioux:2020}, discrete-time flatness is considered
also for non-submersive and hence non-accessible systems. However,
without accessibility, the practical applicability is significantly
reduced.} Geometrically, the system (\ref{eq:sys}) can be interpreted as a
map
\begin{equation}
f:\mathcal{X}\times\mathcal{U}\rightarrow\mathcal{X}\text{\textsuperscript{+}}\label{eq:map_f}
\end{equation}
from a manifold $\mathcal{X}\times\mathcal{U}$ with coordinates $(x,u)$
to a manifold $\mathcal{X}\text{\textsuperscript{+}}$ with coordinates
$x^{+}$. The condition (\ref{eq:submersivity}) ensures that this
map is a submersion and therefore locally surjective. The notation
with a superscript $+$ is used to denote the forward-shift of the
corresponding variable. For the inputs and flat outputs we also need
higher forward-shifts, and use a subscript in brackets. For instance,
$u_{[\alpha]}$ denotes the $\alpha$-th forward-shift of $u$. To
keep formulas short and readable we use the Einstein summation convention,
and in order to avoid mathematical subtleties we also assume that
all functions are smooth. Furthermore, since we use the inverse- and
the implicit function theorem, the flow-box theorem, and the Frobenius
theorem, it is important to emphasize that all our results are only
local. Thus, in order not to loose localness, we have to ensure that
the value of $x^{+}$ determined by (\ref{eq:sys}) is sufficiently
close to $x$. Since the map (\ref{eq:sys}) is continuous, this can
be achieved by considering only a sufficiently small neighborhood
of an equilibrium point $(x_{0},u_{0})$. This is a common practice
in the discrete-time literature, see e.g. \cite{Kotta:1995}. For
instance, the closely related discrete-time static feedback linearization
problem is also considered around an equilibrium point, see \cite{Grizzle:1986}
or \cite{NijmeijervanderSchaft:1990}. For continuous-time systems
such a problem does not appear, since by considering a sufficiently
small time interval it can always be ensured that the solution stays
arbitrarily close to the initial state.

In the following, we summarize the concept of difference flatness.
For this purpose, we introduce a space with coordinates $(x,u,u_{[1]},u_{[2]},\ldots)$
and the forward-shift operator $\delta_{xu}$, which acts on a function
$g$ according to the rule
\[
\delta_{xu}(g(x,u,u_{[1]},u_{[2]},\ldots))=g(f(x,u),u_{[1]},u_{[2]},u_{[3]},\ldots)\,.
\]
A repeated application of $\delta_{xu}$ is denoted by $\delta_{xu}^{\alpha}$.
Since in an equilibrium $(x_{0},u_{0})$ the input is kept constant,
in this framework it corresponds to a point $(x_{0},u_{0},u_{0},u_{0},\ldots)$.
\begin{defn}
\label{def:flatness}The system (\ref{eq:sys}) is said to be flat
around an equilibrium $(x_{0},u_{0})$, if the $n+m$ coordinate functions
$x$ and $u$ can be expressed locally by an $m$-tuple of functions
\begin{equation}
y^{j}=\varphi^{j}(x,u,u_{[1]},\ldots,u_{[q]})\,,\quad j=1,\ldots,m\label{eq:flat_output}
\end{equation}
and their forward-shifts
\[
\begin{array}{ccl}
y_{[1]} & = & \delta_{xu}(\varphi(x,u,u_{[1]},\ldots,u_{[q]}))\\
y_{[2]} & = & \delta_{xu}^{2}(\varphi(x,u,u_{[1]},\ldots,u_{[q]}))\\
 & \vdots
\end{array}
\]
up to some finite order. The $m$-tuple (\ref{eq:flat_output}) is
called a flat output.
\end{defn}

With this definition, flatness is equivalent to endogenous dynamic
feedback linearizability as it is defined in \cite{Aranda-BricaireMoog:2008}.
The representation of $x$ and $u$ by the flat output and its forward-shifts
is unique, and has the form\footnote{The multi-index $R=(r_{1},\ldots,r_{m})$ contains the number of forward-shifts
of each component of the flat output which is needed to express $x$
and $u$, and $y_{[0,R]}$ is an abbreviation for $y$ and its forward-shifts
up to order $R$.}
\begin{equation}
\begin{array}{cclcl}
x^{i} & = & F_{x}^{i}(y_{[0,R-1]})\,, & \quad & i=1,\ldots,n\\
u^{j} & = & F_{u}^{j}(y_{[0,R]})\,, & \quad & j=1,\ldots,m\,.
\end{array}\label{eq:flat_param}
\end{equation}
Further details can be found e.g. in \cite{KolarSchoberlDiwold:2019}.
For the proof of our main result, the necessary and sufficient condition
of Theorem \ref{thm:necessary_and_sufficient_conditions} in Section
\ref{sec:Necessary_and_Sufficient}, we need in particular the following
two lemmas.
\begin{lem}
\label{lem:basic_decomposition_flat} (\cite{KolarKaldmaeSchoberlKottaSchlacher:2016},
\cite{KolarSchoberlDiwold:2019}) A system of the form
\begin{equation}
\begin{array}{ll}
x_{1}^{i_{1},+}=f_{1}^{i_{1}}(x_{1},x_{2},u_{1})\,, & i_{1}=1,\ldots,n-m_{2}\\
x_{2}^{i_{2},+}=f_{2}^{i_{2}}(x_{1},x_{2},u_{1},u_{2})\,,\quad & i_{2}=1,\ldots,m_{2}
\end{array}\label{eq:basic_decomposition_flat}
\end{equation}
with $\dim(u_{2})=\dim(x_{2})=m_{2}$ and $\mathrm{rank}(\partial_{u}f)=\dim(u)=m$
is flat if and only if the subsystem
\begin{equation}
x_{1}^{+}=f_{1}(x_{1},x_{2},u_{1})\label{eq:basic_decomposition_flat_subsys}
\end{equation}
with the $m$ inputs $(x_{2},u_{1})$ is flat.
\end{lem}

The equations $x_{2}^{+}=f_{2}(x_{1},x_{2},u_{1},u_{2})$ of (\ref{eq:basic_decomposition_flat})
can be interpreted as an endogenous dynamic feedback for the subsystem
(\ref{eq:basic_decomposition_flat_subsys}). This is in accordance
with the fact that applying or removing an endogenous dynamic feedback
has no effect on the flatness of a system. It is important to note
that the Jacobian matrix $\partial_{(x_{2},u_{1})}f_{1}$ does not
necessarily have rank $m$, which means that the subsystem (\ref{eq:basic_decomposition_flat_subsys})
may have redundant inputs. However, redundant inputs can be eliminated
by suitable transformations: For a system (\ref{eq:sys}) with $\mathrm{rank}(\partial_{u}f)=\hat{m}<m$,
there always exists an input transformation $(\hat{u},\tilde{u})=\Phi_{u}(x,u)$
with $\dim(\hat{u})=\hat{m}$ that eliminates $m-\hat{m}$ redundant
inputs $\tilde{u}$. There is a simple connection between a flat output
of the transformed system with $\hat{m}$ inputs, and the original
system (\ref{eq:sys}) with $m$ inputs.
\begin{lem}
\label{lem:flatness_redundant_inputs} (\cite{KolarSchoberlDiwold:2019})
Consider a system (\ref{eq:sys}) with $\mathrm{rank}(\partial_{u}f)=\hat{m}<m$,
and an input transformation $(\hat{u},\tilde{u})=\Phi_{u}(x,u)$ with
$\dim(\hat{u})=\hat{m}$ that eliminates $m-\hat{m}$ redundant inputs
$\tilde{u}$. If an $\hat{m}$-tuple $\hat{y}$ is a flat output of
the transformed system
\[
x^{i,+}=\hat{f}^{i}(x,\hat{u})\,,\quad i=1,\ldots,n
\]
with the $\hat{m}$ inputs $\hat{u}$, then the $m$-tuple $y=(\hat{y},\tilde{u})$
is a flat output of the original system (\ref{eq:sys}) with the $m$
inputs $u$.
\end{lem}

Thus, redundant inputs are candidates for components of a flat output.

\section{\label{sec:Projectable_Distributions}Projectable Vector Fields and
Distributions}

The necessary and sufficient conditions for flatness that we derive
in Section \ref{sec:Necessary_and_Sufficient} are based on the concept
of projectable vector fields and distributions. In the following,
we give a brief overview and refer to \cite{KolarSchoberlDiwold:2019}
or \cite{Boothby:1986} for further details.

We call a vector field
\begin{equation}
v=v_{x}^{i}(x,u)\partial_{x^{i}}+v_{u}^{j}(x,u)\partial_{u^{j}}\label{eq:f-related_v}
\end{equation}
on $\mathcal{X}\times\mathcal{U}$ ``projectable'' with respect
to the map (\ref{eq:map_f}), if a pointwise application of the tangent
map $f_{*}:\mathcal{T}(\mathcal{X}\times\mathcal{U})\rightarrow\mathcal{T}(\mathcal{X}^{+})$
yields a well-defined vector field
\begin{equation}
w=w^{i}(x^{+})\partial_{x^{i,+}}\label{eq:f-related_w}
\end{equation}
on $\mathcal{X}^{+}$. In this case, the components of (\ref{eq:f-related_v})
and (\ref{eq:f-related_w}) meet
\[
w^{i}(x^{+})\circ f(x,u)=\partial_{x^{k}}f^{i}v_{x}^{k}(x,u)+\partial_{u^{j}}f^{i}v_{u}^{j}(x,u)\,,
\]
$i=1,\ldots,n$. The vector fields are said to be $f$-related and
we write $w=f_{*}(v)$. Checking whether a vector field (\ref{eq:f-related_v})
is projectable or not is a simple task if we introduce coordinates
\begin{equation}
\begin{array}{ccl}
\theta^{i} & = & f^{i}(x,u)\,,\quad i=1,\ldots,n\\
\xi^{j} & = & h^{j}(x,u)\,,\quad j=1,\ldots,m
\end{array}\label{eq:adapted_coordinates}
\end{equation}
on $\mathcal{X}\times\mathcal{U}$ which are adapted to the ``fibration''
(foliation) determined by the map (\ref{eq:map_f}). The $m$ functions
$h^{j}(x,u)$ must be chosen such that the Jacobian matrix of the
right-hand side of (\ref{eq:adapted_coordinates}) is regular. Because
of the submersivity condition (\ref{eq:submersivity}), this is always
possible and can be achieved e.g. by choosing suitable components
of $x$ or $u$. With coordinates $(\theta,\xi)$ on $\mathcal{X}\times\mathcal{U}$,
the map (\ref{eq:map_f}) has the simple form
\begin{equation}
x^{i,+}=\theta^{i}\,,\quad i=1,\ldots,n\,.\label{eq:f_pr1}
\end{equation}
All points of $\mathcal{X}\times\mathcal{U}$ with the same value
of $\theta$ belong to the same fibre and are mapped to the same point
of $\mathcal{X}^{+}$, regardless of the value of the fibre coordinates
$\xi$. In adapted coordinates, a vector field (\ref{eq:f-related_v})
on $\mathcal{X}\times\mathcal{U}$ has in general the form
\begin{equation}
v=a^{i}(\theta,\xi)\partial_{\theta^{i}}+b^{j}(\theta,\xi)\partial_{\xi^{j}}\,,\label{eq:f-related_v_adapt}
\end{equation}
and because of (\ref{eq:f_pr1}) an application of the tangent map
$f_{*}$ yields
\begin{equation}
f_{*}(v)=a^{i}(\theta,\xi)\partial_{x^{i,+}}\,.\label{eq:f-related_fstar_v_adapt}
\end{equation}
Obviously, the pointwise pushforward (\ref{eq:f-related_fstar_v_adapt})
of (\ref{eq:f-related_v_adapt}) induces a well-defined vector field
on $\mathcal{X}^{+}$ if and only if the functions $a^{i}$ are independent
of the coordinates $\xi$. In this case, replacing $\theta$ by $x^{+}$
yields the vector field (\ref{eq:f-related_w}). In summary, a vector
field (\ref{eq:f-related_v}) is projectable if and only if in adapted
coordinates (\ref{eq:adapted_coordinates}) it has the form
\begin{equation}
a^{i}(\theta)\partial_{\theta^{i}}+b^{j}(\theta,\xi)\partial_{\xi^{j}}\,,\label{eq:projectable_vector_field_adapt}
\end{equation}
and the corresponding vector field (\ref{eq:f-related_w}) is given
by
\[
a^{i}(x^{+})\partial_{x^{i,+}}\,.
\]

Similar to projectable vector fields, we call a distribution $D$
on $\mathcal{X}\times\mathcal{U}$ ``projectable'' if it admits
a basis that consists of projectable vector fields (not every vector
field contained in a projectable distribution has to be projectable
itself). The existence of such a basis ensures that the pushforward
$f_{*}(D)$ of a projectable distribution is a well-defined distribution
on $\mathcal{X}^{+}$. Moreover, since the Lie brackets $[v_{1},v_{2}]$
and $[w_{1},w_{2}]$ of two pairs $v_{1},w_{1}$ and $v_{2},w_{2}$
of $f$-related vector fields are again $f$-related, i.e. $f_{*}[v_{1},v_{2}]=[w_{1},w_{2}]$,
the pushforward of an involutive projectable distribution is again
an involutive distribution. The following theorem is essential for
the uniqueness of the sequence of distributions introduced in Section
\ref{sec:Necessary_and_Sufficient}.
\begin{thm}
\label{thm:largest_projectable_subdistribution}The largest projectable
subdistribution $D\subset E$ of a distribution $E$ on $\mathcal{X}\times\mathcal{U}$
is uniquely determined. If $E$ is involutive, then $D$ is also involutive.
\end{thm}

\begin{proof}
Introduce adapted coordinates (\ref{eq:adapted_coordinates}) on $\mathcal{X}\times\mathcal{U}$,
and construct a new basis for $E$ which contains as many projectable
vector fields (\ref{eq:projectable_vector_field_adapt}) as possible.
These projectable vector fields are a basis for the largest projectable
subdistribution $D$. Every other projectable vector field in $E$
can be written as linear combination of these vector fields. If the
distribution $E$ is involutive, then all pairwise Lie brackets of
the basis vector fields of the projectable subdistribution $D\subset E$
must be contained in $E$. However, since the basis vector fields
of $D$ are projectable, the Lie brackets are again projectable vector
fields (see above). Since by construction $D$ contains all projectable
vector fields of $E$, the subdistribution $D$ is involutive itself.
\end{proof}
In adapted coordinates (\ref{eq:adapted_coordinates}), the computation
of the largest projectable subdistribution of a given $d$-dimensional
distribution requires only algebraic manipulations. It is convenient
to introduce in a first step a ``normalized'' basis consisting of
vector fields
\begin{align}
v_{k} & =\partial_{\theta^{k}}+a_{k}^{\bar{d}+1}(\theta,\xi)\partial_{\theta^{\bar{d}+1}}+\ldots+a_{k}^{n}(\theta,\xi)\partial_{\theta^{n}}\label{eq:normalized_basis}\\
 & \hphantom{\phantom{=\partial_{\theta^{1}}}}+b_{k}^{1}(\theta,\xi)\partial_{\xi^{1}}+\ldots+b_{k}^{m}(\theta,\xi)\partial_{\xi^{m}}\,,\nonumber 
\end{align}
$k=1,\ldots,\bar{d}$ and
\begin{equation}
v_{l}=b_{l}^{1}(\theta,\xi)\partial_{\xi^{1}}+\ldots+b_{l}^{m}(\theta,\xi)\partial_{\xi^{m}}\,,\label{eq:normalized_basis_proj_to_0}
\end{equation}
$l=\bar{d}+1,\ldots,d$, where $d-\bar{d}$ is the dimension of the
subdistribution projecting to zero. Obviously, the $d-\bar{d}$ vector
fields (\ref{eq:normalized_basis_proj_to_0}) are already projectable.
As shown in the appendix of \cite{KolarSchoberlDiwold:2019}, a linear
combination
\begin{equation}
c^{1}(\theta,\xi)v_{1}+\ldots+c^{\bar{d}}(\theta,\xi)v_{\bar{d}}\label{eq:projectable_linear_combination}
\end{equation}
of the $\bar{d}$ remaining vector fields (\ref{eq:normalized_basis})
is projectable if and only if the coefficients $c^{1},\ldots,c^{\bar{d}}$
are independent of $\xi$ and meet
\[
\partial_{\xi^{j}}a_{k}^{i}(\theta,\xi)c^{k}(\theta)=0\quad\forall i=\bar{d}+1,\ldots,n\:\text{and}\:j=1,\ldots,m\,.
\]
In other words, a linear combination (\ref{eq:projectable_linear_combination})
is projectable if and only if the column vector of coefficients
\begin{equation}
\begin{bmatrix}c^{1}(\theta) & \cdots & c^{\bar{d}}(\theta)\end{bmatrix}^{T}\label{eq:coefficient_vector}
\end{equation}
lies in the kernel of the $(n-\bar{d})m\times\bar{d}$-matrix
\begin{equation}
\begin{bmatrix}\partial_{\xi^{j}}a_{1}^{i}(\theta,\xi) & \cdots & \partial_{\xi^{j}}a_{\bar{d}}^{i}(\theta,\xi)\end{bmatrix}\label{eq:matrix_kernel}
\end{equation}
with $i=\bar{d}+1,\ldots,n$ and $j=1,\ldots,m$. In contrast to \cite{KolarSchoberlDiwold:2019},
we want to find a maximal set of independent solutions (\ref{eq:coefficient_vector}).
This can be achieved by a successive reduction strategy: If we bring
the matrix (\ref{eq:matrix_kernel}) with row manipulations into the
form
\[
\begin{bmatrix}I & R(\theta,\xi)\\
0 & 0
\end{bmatrix}
\]
with an identity matrix $I$ and some remaining matrix $R(\theta,\xi)$,
(\ref{eq:coefficient_vector}) can only lie in its kernel if the product
of $R(\theta,\xi)$ with the corresponding elements of (\ref{eq:coefficient_vector})
is independent of $\xi$. The other elements of (\ref{eq:coefficient_vector}),
i.e., those that are multiplied with the identity matrix, are determined
by the fact that the sum of the two products must vanish. Obviously,
the product with $R(\theta,\xi)$ is independent of $\xi$ if and
only if the corresponding part of (\ref{eq:coefficient_vector}) lies
in the kernel of the matrix
\begin{equation}
\begin{bmatrix}\partial_{\xi^{1}}R(\theta,\xi)\\
\vdots\\
\partial_{\xi^{m}}R(\theta,\xi)
\end{bmatrix}\,.\label{eq:matrix_kernel_2}
\end{equation}
Thus, we have reduced the original problem with the matrix (\ref{eq:matrix_kernel})
to a smaller one with the matrix (\ref{eq:matrix_kernel_2}). Continuing
the procedure yields a maximal set of independent solutions (\ref{eq:coefficient_vector}).
The corresponding linear combinations (\ref{eq:projectable_linear_combination})
form a basis for the largest projectable subdistribution.

\section{\label{sec:Necessary_and_Sufficient}Necessary and Sufficient Conditions}

Except for special situations as in \cite{MilleriouxJungers:2021}
with switched systems, verifying whether a given output (\ref{eq:flat_output})
is a flat output or not is in principle a straightforward task (with
a restriction to some upper limit for the number of shifts in (\ref{eq:flat_param})).
In the following, we address the significantly more difficult problem
of checking whether there exists a flat output or not.

In \cite{KolarSchoberlDiwold:2019} it has been shown that every flat
discrete-time system satisfies the following necessary condition.
\begin{thm}
\label{thm:necessary_condition_flat}The input distribution $\mathrm{span}\{\partial_{u}\}$
of a flat system (\ref{eq:sys}) with $\mathrm{rank}(\partial_{u}f)=m$
contains a nontrivial projectable vector field.
\end{thm}

In other words, the input distribution contains an at least 1-dimensional
projectable subdistribution. Such a projectable vector field or subdistribution
can be used to transform the system into the decomposed form (\ref{eq:basic_decomposition_flat})
of Lemma \ref{lem:basic_decomposition_flat}, where the complete system
is flat if and only if the subsystem (\ref{eq:basic_decomposition_flat_subsys})
is flat. The required input- and state transformations can be constructed
by straightening out the vector field and its pushforward (or the
corresponding distributions) by the flow-box theorem / Frobenius theorem.
As shown in \cite{KolarSchoberlDiwold:2019}, a repeated application
of this decomposition allows to check whether a system (\ref{eq:sys})
is flat or not in at most $n-1$ steps. However, straightening out
vector fields / distributions with the flow-box theorem / Frobenius
theorem requires the solution of ODEs. For this reason, in the following
we introduce a computationally more efficient a priori test, which
allows to check whether the repeated decompositions are possible or
not without actually performing them. This test relies on sequences
of nested distributions on $\mathcal{X}\times\mathcal{U}$ and $\mathcal{X}\text{\textsuperscript{+}}$.
The construction of these sequences of distributions is based on the
map (\ref{eq:map_f}) defined by the system equations (\ref{eq:sys}),
and the map
\[
\pi:\mathcal{X}\times\mathcal{U}\rightarrow\mathcal{X}^{+}
\]
defined by
\[
x^{i,+}=x^{i}\,,\quad i=1,\ldots,n\,.
\]
We assume that the distributions have locally constant dimension.\textbf{\smallskip{}
}\\
\textbf{Algorithm 1}\textbf{\emph{.\smallskip{}
}}\\
\textbf{\emph{Step $0$:}}\emph{ Define the involutive distributions
$\Delta_{0}=0$ on $\mathcal{X}^{+}$ and}\footnote{As in \cite{Grizzle:1986} and \cite{NijmeijervanderSchaft:1990},
$\pi_{*}^{-1}(\Delta)$ denotes the inverse image of a distribution
$\Delta$ under the tangent map $\pi_{*}$. If $\Delta=\mathrm{span}\{v_{1}^{i}(x^{+})\partial_{x^{i,+}},\ldots,v_{d}^{i}(x^{+})\partial_{x^{i,+}}\}$
is a $d$-dimensional distribution on $\mathcal{X}^{+}$, then $\pi_{*}^{-1}(\Delta)$
is the $(d+m)$-dimensional distribution $\mathrm{span}\{v_{1}^{i}(x)\partial_{x^{i}},\ldots,v_{d}^{i}(x)\partial_{x^{i}},\partial_{u^{1}},\ldots,\partial_{u^{m}}\}$
on $\mathcal{X}\times\mathcal{U}$.}\emph{ $E_{0}=\pi_{*}^{-1}(\Delta_{0})=\mathrm{span}\{\partial_{u}\}$
on $\mathcal{X}\times\mathcal{U}$. Then compute the largest subdistribution
$D_{0}\subset E_{0}$ which is projectable with respect to the map
(\ref{eq:map_f}). Because of Theorem \ref{thm:largest_projectable_subdistribution},
$D_{0}$ is unique and involutive. Thus, the pushforward $\Delta_{1}=f_{*}(D_{0})$
is a well-defined involutive distribution on $\mathcal{X}^{+}$.}\textbf{\emph{}}\\
\textbf{\emph{Step $k\geq1$:}}\emph{ Define the involutive distribution
\[
E_{k}=\pi_{*}^{-1}(\Delta_{k})
\]
on $\mathcal{X}\times\mathcal{U}$. Because of $\Delta_{k-1}\subset\Delta_{k}$,
it satisfies $E_{k-1}\subset E_{k}$. Then compute the largest subdistribution
\[
D_{k}\subset E_{k}
\]
which is projectable with respect to the map (\ref{eq:map_f}). Because
of Theorem \ref{thm:largest_projectable_subdistribution}, $D_{k}$
is unique and involutive. Moreover, because of $D_{k-1}\subset E_{k-1}$
and $E_{k-1}\subset E_{k}$, it satisfies $D_{k-1}\subset D_{k}$.
Thus, the pushforward
\begin{equation}
\Delta_{k+1}=f_{*}(D_{k})\label{eq:Delta_fstar_D}
\end{equation}
is a well-defined involutive distribution on $\mathcal{X}^{+}$ with
\begin{equation}
\Delta_{k}\subset\Delta_{k+1}\,.\label{eq:dim_deltak1_dim_deltak}
\end{equation}
}\textbf{\emph{Stop}}\emph{ if $\dim(\Delta_{\bar{k}+1})=\dim(\Delta_{\bar{k}})$
for some $k=\bar{k}$.}
\begin{rem}
The distributions $\Delta_{k}$ are involutive since the pushforward
of a projectable and involutive distribution $D_{k-1}$ is again an
involutive distribution. Subsequently, the involutivity of $\Delta_{k}$
implies the involutivity of $E_{k}$. For $k=0$ with $E_{0}=\mathrm{span}\{\partial_{u}\}$,
this is obvious. For $k\geq1$ we know that $\Delta_{k}$ is involutive,
and can perform a state transformation $(\tilde{x}_{1},\tilde{x}_{2})=\Phi_{x}(x)$
with $\dim(\tilde{x}_{1})=\dim(\Delta_{k})$ such that $\Delta_{k}=\mathrm{span}\{\partial_{\tilde{x}_{1}^{+}}\}$.
In these coordinates, $E_{k}=\mathrm{span}\{\partial_{\tilde{x}_{1}},\partial_{u}\}$
is straightened out and hence clearly involutive.
\end{rem}

Because of (\ref{eq:dim_deltak1_dim_deltak}) and $\dim(\mathcal{X}^{+})=n$,
the procedure terminates after at most $n$ steps. It yields a unique
nested sequence of projectable and involutive distributions
\begin{equation}
D_{0}\subset D_{1}\subset\ldots\subset D_{\bar{k}-1}\label{eq:chain_D}
\end{equation}
on $\mathcal{X}\times\mathcal{U}$, and a unique nested sequence of
involutive distributions
\begin{equation}
\Delta_{1}\subset\Delta_{2}\subset\ldots\subset\Delta_{\bar{k}}\label{eq:chain_Delta}
\end{equation}
on $\mathcal{X}^{+}$, which are related by the condition (\ref{eq:Delta_fstar_D}).
Since we have assumed that locally all these distributions have constant
dimension, we can define
\[
\rho_{k}=\dim(\Delta_{k})-\dim(\Delta_{k-1})\,,\quad k\geq1
\]
with $\dim(\Delta_{0})=0$. Since the pushforward of linearly independent,
projectable vector fields on $\mathcal{X}\times\mathcal{U}$ does
not necessarily yield linearly independent vector fields on $\mathcal{X}^{+}$,
it is also important to note that in general
\[
\dim(\Delta_{k})\leq\dim(D_{k-1})\,.
\]
Therefore, we additionally define
\begin{equation}
\mu_{k}=\dim(D_{k})-\dim(\Delta_{k+1})-\underbrace{\left(\dim(D_{k-1})-\dim(\Delta_{k})\right)}_{\mu_{0}+\ldots+\mu_{k-1}}\,,\label{eq:mu_k}
\end{equation}
$k\geq1$, which is just the number of linearly independent vector
fields $v\in D_{k}$ with $f_{*}(v)=0$ that are not contained in
$D_{k-1}$. For $k=0$, in the case $\mathrm{rank}(\partial_{u}f)=m$
we always have
\[
\mu_{0}=\dim(D_{0})-\dim(\Delta_{1})=0\,.
\]

The sequence (\ref{eq:chain_Delta}) generalizes a sequence which
was introduced in \cite{Grizzle:1986} to check whether a discrete-time
system (\ref{eq:sys}) is static feedback linearizable or not.
\begin{thm}
A system (\ref{eq:sys}) with $\mathrm{rank}(\partial_{u}f)=m$ is
static feedback linearizable if and only if $D_{k}=E_{k}$, $k\geq0$
and $\dim(\Delta_{\bar{k}})=n$.
\end{thm}

Hence, for static feedback linearizability, in every step of Algorithm
1 the complete distribution $E_{k}$ must be projectable. For a proof
see \cite{Grizzle:1986} or \cite{NijmeijervanderSchaft:1990}. If
we drop the condition $D_{k}=E_{k}$, we get necessary and sufficient
conditions for flatness.
\begin{thm}
\label{thm:necessary_and_sufficient_conditions}A system (\ref{eq:sys})
with $\mathrm{rank}(\partial_{u}f)=m$ is flat if and only if $\dim(\Delta_{\bar{k}})=n$.
\end{thm}

Before we prove Theorem \ref{thm:necessary_and_sufficient_conditions},
we establish some further properties of the sequences (\ref{eq:chain_D})
and (\ref{eq:chain_Delta}). The basic idea of the proof, however,
is to use the distributions (\ref{eq:chain_D}) and (\ref{eq:chain_Delta})
for a stepwise decomposition of the system (\ref{eq:sys}) into subsystems
and endogenous dynamic feedbacks exactly like in \cite{KolarSchoberlDiwold:2019}.
In the case $\dim(\Delta_{\bar{k}})=n$, the system can be decomposed
until only the trivial system is left, which proves that the system
is flat. In the case $\dim(\Delta_{\bar{k}})<n$, in contrast, there
occurs a subsystem which allows no further decomposition. This is,
however, a contradiction to the necessary condition for flatness derived
in \cite{KolarSchoberlDiwold:2019}.

First, it is important to note that the nested sequence of involutive
distributions (\ref{eq:chain_Delta}) on $\mathcal{X}^{+}$ can be
straightened out by a state transformation
\begin{equation}
(\bar{x}_{1},\ldots,\bar{x}_{\bar{k}},\bar{x}_{rest})=\Phi_{x}(x)\label{eq:state_transformation_Delta}
\end{equation}
with $\dim(\bar{x}_{k})=\rho_{k}$, $k=1,\ldots,\bar{k}$ such that
\[
\begin{array}{ccl}
\Delta_{1} & = & \mathrm{span}\{\partial_{\bar{x}_{1}^{+}}\}\\
\Delta_{2} & = & \mathrm{span}\{\partial_{\bar{x}_{1}^{+}},\partial_{\bar{x}_{2}^{+}}\}\\
 & \vdots\\
\Delta_{\bar{k}} & = & \mathrm{span}\{\partial_{\bar{x}_{1}^{+}},\partial_{\bar{x}_{2}^{+}},\ldots,\partial_{\bar{x}_{\bar{k}}^{+}}\}\,.
\end{array}
\]
In accordance with the transformation law for discrete-time systems,
the state transformation (\ref{eq:state_transformation_Delta}) is
performed both for the variables $x$ and the shifted variables $x^{+}$.
Because of (\ref{eq:Delta_fstar_D}), the transformed system
\begin{equation}
\begin{array}{ccl}
\bar{x}_{rest}^{+} & = & f_{rest}(\bar{x},u)\\
\bar{x}_{\bar{k}}^{+} & = & f_{\bar{k}}(\bar{x},u)\\
 & \vdots\\
\bar{x}_{2}^{+} & = & f_{2}(\bar{x},u)\\
\bar{x}_{1}^{+} & = & f_{1}(\bar{x},u)
\end{array}\label{eq:sys_straightened_out_Delta}
\end{equation}
meets
\begin{equation}
\begin{array}{ccl}
f_{*}(D_{0}) & = & \mathrm{span}\{\partial_{\bar{x}_{1}^{+}}\}\\
f_{*}(D_{1}) & = & \mathrm{span}\{\partial_{\bar{x}_{1}^{+}},\partial_{\bar{x}_{2}^{+}}\}\\
 & \vdots\\
f_{*}(D_{\bar{k}-1}) & = & \mathrm{span}\{\partial_{\bar{x}_{1}^{+}},\partial_{\bar{x}_{2}^{+}},\ldots,\partial_{\bar{x}_{\bar{k}}^{+}}\}\,.
\end{array}\label{eq:pushforwards_chain_D}
\end{equation}
In these coordinates, the involutive distributions
\[
E_{k}=\pi_{*}^{-1}(\Delta_{k})=\mathrm{span}\{\partial_{\bar{x}_{k}},\ldots,\partial_{\bar{x}_{1}},\partial_{u}\}\,,
\]
$k=0,\ldots,\bar{k}-1$ are exactly the input distributions of the
subsystems
\begin{equation}
\begin{array}{ccl}
\bar{x}_{rest}^{+} & = & f_{rest}(\bar{x}_{rest},\bar{x}_{\bar{k}},\ldots,\bar{x}_{1},u)\\
\bar{x}_{\bar{k}}^{+} & = & f_{\bar{k}}(\bar{x}_{rest},\bar{x}_{\bar{k}},\ldots,\bar{x}_{1},u)\\
 & \vdots\\
\bar{x}_{k+1}^{+} & = & f_{k+1}(\bar{x}_{rest},\bar{x}_{\bar{k}},\ldots,\bar{x}_{1},u)
\end{array}\label{eq:sys_straightened_out_Delta_subsys}
\end{equation}
of (\ref{eq:sys_straightened_out_Delta}) without the equations for
$(\bar{x}_{k},\ldots,\bar{x}_{1})$. Among the inputs $(\bar{x}_{k},\ldots,\bar{x}_{1},u)$
of these subsystems there are of course redundant inputs.
\begin{lem}
\label{lem:jacobian_matrix_subsys_rank}The rank of the Jacobian matrix
\begin{equation}
\left[\begin{array}{cccc}
\partial_{\bar{x}_{k}}f_{rest} & \cdots & \partial_{\bar{x}_{1}}f_{rest} & \partial_{u}f_{rest}\\
\partial_{\bar{x}_{k}}f_{\bar{k}} & \cdots & \partial_{\bar{x}_{1}}f_{\bar{k}} & \partial_{u}f_{\bar{k}}\\
\vdots &  & \vdots & \vdots\\
\partial_{\bar{x}_{k}}f_{k+1} & \cdots & \partial_{\bar{x}_{1}}f_{k+1} & \partial_{u}f_{k+1}
\end{array}\right]\label{eq:jacobian_matrix_subsys}
\end{equation}
of the subsystem (\ref{eq:sys_straightened_out_Delta_subsys}) with
respect to its inputs $(\bar{x}_{k},\ldots,\bar{x}_{1},u)$ is given
by
\begin{equation}
m-\left(\dim(D_{k})-\dim(\Delta_{k+1})\right)=m-(\mu_{0}+\ldots+\mu_{k})\,.\label{eq:jacobian_matrix_subsys_rank}
\end{equation}
\end{lem}

\begin{proof}
The Jacobian matrix (\ref{eq:jacobian_matrix_subsys}) has
\[
\rho_{k}+\ldots+\rho_{1}+m=\dim(\Delta_{k})+m
\]
columns. To prove the lemma, we simply calculate the dimension of
its kernel. For the distribution $D_{k}\subset E_{k}$ with $f_{*}(D_{k})=\Delta_{k+1}$,
there exists a basis that contains $\dim(D_{k})-\rho_{k+1}$ vector
fields with a pushforward that lies in $\Delta_{k}\subset\Delta_{k+1}$.
Written as column vectors, these $\dim(D_{k})-\rho_{k+1}$ vector
fields lie in the kernel of the Jacobian matrix (\ref{eq:jacobian_matrix_subsys}).
Thus, the matrix has a kernel of dimension at least $\dim(D_{k})-\rho_{k+1}$.
However, every vector field $v\in E_{k}$ which lies in the kernel
of (\ref{eq:jacobian_matrix_subsys}) has a pushforward that lies
in $\Delta_{k}$, and because of $f_{*}(D_{k-1})=\Delta_{k}$ it is
certainly contained in $D_{k}$. Thus, the dimension of the kernel
is exactly $\dim(D_{k})-\rho_{k+1}$. Subtracting the dimension of
the kernel from the number of columns gives the rank
\[
\dim(\Delta_{k})+m-\left(\dim(D_{k})-\rho_{k+1}\right)\,.
\]
With $\rho_{k+1}=\dim(\Delta_{k+1})-\dim(\Delta_{k})$ and the definition
(\ref{eq:mu_k}) of the integers $\mu_{k}$, the result (\ref{eq:jacobian_matrix_subsys_rank})
follows.
\end{proof}
If the system (\ref{eq:sys}) is static feedback linearizable, then
the state transformation (\ref{eq:state_transformation_Delta}) transforms
the system into a triangular form
\begin{equation}
\begin{array}{ccl}
\bar{x}_{\bar{k}}^{+} & = & f_{\bar{k}}(\bar{x}_{\bar{k}},\bar{x}_{\bar{k}-1})\\
 & \vdots\\
\bar{x}_{2}^{+} & = & f_{2}(\bar{x}_{\bar{k}},\bar{x}_{\bar{k}-1},\ldots,\bar{x}_{1})\\
\bar{x}_{1}^{+} & = & f_{1}(\bar{x}_{\bar{k}},\bar{x}_{\bar{k}-1},\ldots,\bar{x}_{1},u)\,,
\end{array}\label{eq:triangular_form_static_feedback_lin}
\end{equation}
see also \cite{Grizzle:1986} and \cite{NijmeijervanderSchaft:1990}.
The reason is that with $D_{k}=E_{k}$, $k\geq0$ straightening out
the sequence (\ref{eq:chain_Delta}) simultaneously straightens out
the sequence $E_{0}\subset E_{1}\subset\ldots\subset E_{\bar{k}-1}$.
Consequently, $D_{k}=\mathrm{span}\{\partial_{\bar{x}_{k}},\ldots,\partial_{\bar{x}_{1}},\partial_{u}\}$,
$k=0,\ldots,\bar{k}-1$. Evaluating the condition (\ref{eq:pushforwards_chain_D})
shows that the transformed system (\ref{eq:sys_straightened_out_Delta})
with $\dim(\bar{x}_{rest})=0$ must have the triangular form (\ref{eq:triangular_form_static_feedback_lin}).
However, if $D_{k}=E_{k}$ does not hold for all $k=0,\ldots,\bar{k}-1$,
then the state transformation (\ref{eq:state_transformation_Delta})
that straightens out the sequence (\ref{eq:chain_Delta}) does in
general not straighten out the sequence (\ref{eq:chain_D}).

\subsection{Proof of Theorem \ref{thm:necessary_and_sufficient_conditions}}

In the following, we prove Theorem \ref{thm:necessary_and_sufficient_conditions}
for the system (\ref{eq:sys_straightened_out_Delta}) after the state
transformation (\ref{eq:state_transformation_Delta}), i.e., we assume
that the distributions (\ref{eq:chain_Delta}) have already been straightened
out. The idea of the proof is to straighten out the sequence (\ref{eq:chain_D})
in $\bar{k}$ steps with coordinate transformations on $\mathcal{X}\times\mathcal{U}$
that can be interpreted as input transformations for the subsystems
(\ref{eq:sys_straightened_out_Delta_subsys}), $k=0,\ldots,\bar{k}-1$.
These transformations result in a sequence of decomposed subsystems,
and show inductively that the complete system (\ref{eq:sys_straightened_out_Delta})
is flat if and only if the subsystem with $f_{rest}$ is flat.

First, we decompose the system (\ref{eq:sys_straightened_out_Delta})
by straightening out $D_{0}$. Because of $D_{0}\subset E_{0}=\mathrm{span}\{\partial_{u}\}$,
there exists an input transformation
\[
(\eta_{0},\hat{z}_{0})=\Phi_{0}(\bar{x},u)
\]
with inverse $u=\hat{\Phi}_{0}(\bar{x},\eta_{0},\hat{z}_{0})$ such
that
\[
D_{0}=\mathrm{span}\{\partial_{\hat{z}_{0}}\}\,.
\]
Because of $f_{*}(D_{0})=\mathrm{span}\{\partial_{\bar{x}_{1}^{+}}\}$,
in these coordinates the functions $f_{2},\ldots,f_{\bar{k}},f_{rest}$
are independent of $\hat{z}_{0}$, i.e.,
\[
\begin{array}{ccl}
\bar{x}_{rest}^{+} & = & f_{rest}(\bar{x}_{rest},\bar{x}_{\bar{k}},\ldots,\bar{x}_{2},\bar{x}_{1},\eta_{0})\\
\bar{x}_{\bar{k}}^{+} & = & f_{\bar{k}}(\bar{x}_{rest},\bar{x}_{\bar{k}},\ldots,\bar{x}_{2},\bar{x}_{1},\eta_{0})\\
 & \vdots\\
\bar{x}_{2}^{+} & = & f_{2}(\bar{x}_{rest},\bar{x}_{\bar{k}},\ldots,\bar{x}_{2},\bar{x}_{1},\eta_{0})\\
\bar{x}_{1}^{+} & = & f_{1}(\bar{x}_{rest},\bar{x}_{\bar{k}},\ldots,\bar{x}_{2},\bar{x}_{1},\eta_{0},\hat{z}_{0})\,.
\end{array}
\]
Since the system (\ref{eq:sys_straightened_out_Delta}) does not have
redundant inputs,
\[
\mathrm{rank}(\partial_{\hat{z}_{0}}f_{1})=\dim(\hat{z}_{0})=\rho_{1}
\]
holds. Thus, the system is flat if and only if the subsystem $f_{2},\ldots,f_{\bar{k}},f_{rest}$
with the inputs $(\bar{x}_{1},\eta_{0})$ is flat (cf. Lemma \ref{lem:basic_decomposition_flat}).
Next, because of the rank condition of Lemma \ref{lem:jacobian_matrix_subsys_rank},
we can eliminate $\mu_{1}$ redundant inputs $y_{1}$ of this subsystem
by an input transformation
\[
(\zeta_{1},y_{1})=\Psi_{1}(\bar{x}_{rest},\bar{x}_{\bar{k}},\ldots,\bar{x}_{2},\bar{x}_{1},\eta_{0})
\]
with inverse $(\bar{x}_{1},\eta_{0})=\hat{\Psi}_{1}(\bar{x}_{rest},\bar{x}_{\bar{k}},\ldots,\bar{x}_{2},\zeta_{1},y_{1})$.
This yields
\[
\begin{array}{ccl}
\bar{x}_{rest}^{+} & = & f_{rest}(\bar{x}_{rest},\bar{x}_{\bar{k}},\ldots,\bar{x}_{2},\zeta_{1})\\
\bar{x}_{\bar{k}}^{+} & = & f_{\bar{k}}(\bar{x}_{rest},\bar{x}_{\bar{k}},\ldots,\bar{x}_{2},\zeta_{1})\\
 & \vdots\\
\bar{x}_{2}^{+} & = & f_{2}(\bar{x}_{rest},\bar{x}_{\bar{k}},\ldots,\bar{x}_{2},\zeta_{1})\\
\bar{x}_{1}^{+} & = & f_{1}(\bar{x}_{rest},\bar{x}_{\bar{k}},\ldots,\bar{x}_{2},\zeta_{1},y_{1},\hat{z}_{0})
\end{array}
\]
with $\dim(\zeta_{1})=m-(\dim(D_{1})-\dim(\Delta_{2}))$ according
to (\ref{eq:jacobian_matrix_subsys_rank}). Now that all redundant
inputs are eliminated, we decompose the subsystem $f_{2},\ldots,f_{\bar{k}},f_{rest}$
by straightening out
\[
D_{1}\subset E_{1}=\mathrm{span}\{\partial_{\hat{z}_{0}},\partial_{y_{1}},\partial_{\zeta_{1}}\}\,.
\]
Because of $D_{0}\subset D_{1}$ and $f_{*}(\partial_{y_{1}})\subset\mathrm{span}\{\partial_{\bar{x}_{1}^{+}}\}$
we have
\[
\mathrm{span}\{\partial_{\hat{z}_{0}},\partial_{y_{1}}\}\subset D_{1}\,,
\]
and consequently $D_{1}$ has a basis of the form
\begin{align*}
\partial_{\hat{z}_{0}^{i_{0}}}\,, & \quad i_{0}=1,\ldots,\rho_{1}\\
\partial_{y_{1}^{j_{1}}}\,, & \quad j_{1}=1,\ldots,\mu_{1}\\
\sum_{l=1}^{\dim(\zeta_{1})}v_{i_{1}}^{l}(\bar{x}_{rest},\ldots,\bar{x}_{2},\zeta_{1},y_{1},\hat{z}_{0})\partial_{\zeta_{1}^{l}}\,, & \quad i_{1}=1,\ldots,\rho_{2}\,.
\end{align*}
Up to a renumbering of the components of $\zeta_{1}$, there even
exists a basis
\begin{align*}
\partial_{\hat{z}_{0}^{i_{0}}}\,, & \quad i_{0}=1,\ldots,\rho_{1}\\
\partial_{y_{1}^{j_{1}}}\,, & \quad j_{1}=1,\ldots,\mu_{1}\\
\partial_{\zeta_{1}^{i_{1}}}+\sum_{l=\rho_{2}+1}^{\dim(\zeta_{1})}\hat{v}_{i_{1}}^{l}(\bar{x}_{rest},\ldots,\bar{x}_{2},\zeta_{1})\partial_{\zeta_{1}^{l}}\,, & \quad i_{1}=1,\ldots,\rho_{2}\,,
\end{align*}
which, written in matrix form, contains a block with an identity matrix.
Thus, the involutivity of $D_{1}$ implies that all pairwise Lie brackets
of the vector fields vanish. This in turn implies that the coefficients
of the last $\rho_{2}$ vector fields are independent of $\hat{z}_{0}$
and $y_{1}$. Therefore, these vector fields can be straightened out
by a transformation of the form
\[
(\eta_{1},\hat{z}_{1})=\Phi_{1}(\bar{x}_{rest},\bar{x}_{\bar{k}},\ldots,\bar{x}_{2},\zeta_{1})
\]
with inverse $\zeta_{1}=\hat{\Phi}_{1}(\bar{x}_{rest},\bar{x}_{\bar{k}},\ldots,\bar{x}_{2},\eta_{1},\hat{z}_{1})$,
which can be interpreted as an input transformation for the subsystem
$f_{2},\ldots,f_{\bar{k}},f_{rest}$. In new coordinates we have
\[
D_{1}=\mathrm{span}\{\partial_{\hat{z}_{0}},\partial_{y_{1}},\partial_{\hat{z}_{1}}\}\,,
\]
and because of $f_{*}(D_{1})=\mathrm{span}\{\partial_{\bar{x}_{1}^{+}},\partial_{\bar{x}_{2}^{+}}\}$
the functions $f_{3},\ldots,f_{\bar{k}},f_{rest}$ are independent
of $\hat{z}_{0}$, $y_{1}$, and $\hat{z}_{1}$. Thus,
\[
\begin{array}{ccl}
\bar{x}_{rest}^{+} & = & f_{rest}(\bar{x}_{rest},\bar{x}_{\bar{k}},\ldots,\bar{x}_{2},\eta_{1})\\
\bar{x}_{\bar{k}}^{+} & = & f_{\bar{k}}(\bar{x}_{rest},\bar{x}_{\bar{k}},\ldots,\bar{x}_{2},\eta_{1})\\
 & \vdots\\
\bar{x}_{3}^{+} & = & f_{3}(\bar{x}_{rest},\bar{x}_{\bar{k}},\ldots,\bar{x}_{2},\eta_{1})\\
\bar{x}_{2}^{+} & = & f_{2}(\bar{x}_{rest},\bar{x}_{\bar{k}},\ldots,\bar{x}_{2},\eta_{1},\hat{z}_{1})\\
\bar{x}_{1}^{+} & = & f_{1}(\bar{x}_{rest},\bar{x}_{\bar{k}},\ldots,\bar{x}_{2},\eta_{1},\hat{z}_{1},y_{1},\hat{z}_{0})
\end{array}
\]
and
\[
\mathrm{rank}(\partial_{\hat{z}_{1}}f_{2})=\dim(\hat{z}_{1})=\rho_{2}\,.
\]
Consequently, the system is flat if and only if the subsystem $f_{3},\ldots,f_{\bar{k}},f_{rest}$
with the inputs $(\bar{x}_{2},\eta_{1})$ is flat. In the following
steps, we proceed analogously. First, we eliminate all redundant inputs
of the subsystem $f_{k+1},\ldots,f_{\bar{k}},f_{rest}$, $k\geq2$.
Subsequently, we decompose the subsystem by straightening out the
distribution $D_{k}$ with a transformation that can be interpreted
as an input transformation for the subsystem. Continuing this procedure
until $k=\bar{k}-1$ introduces new coordinates on $\mathcal{X}\times\mathcal{U}$
such that the map (\ref{eq:sys_straightened_out_Delta}) has the form
\[
\begin{array}{ccl}
\bar{x}_{rest}^{+} & = & f_{rest}(\bar{x}_{rest},\bar{x}_{\bar{k}},\eta_{\bar{k}-1})\\
\bar{x}_{\bar{k}}^{+} & = & f_{\bar{k}}(\bar{x}_{rest},\bar{x}_{\bar{k}},\eta_{\bar{k}-1},\hat{z}_{\bar{k}-1})\\
 & \vdots\\
\bar{x}_{2}^{+} & = & f_{2}(\bar{x}_{rest},\bar{x}_{\bar{k}},\eta_{\bar{k}-1},\hat{z}_{\bar{k}-1},y_{\bar{k}-1},\ldots,\hat{z}_{2},y_{2},\hat{z}_{1})\\
\bar{x}_{1}^{+} & = & f_{1}(\bar{x}_{rest},\bar{x}_{\bar{k}},\eta_{\bar{k}-1},\hat{z}_{\bar{k}-1},y_{\bar{k}-1},\ldots,\hat{z}_{1},y_{1},\hat{z}_{0})\,,
\end{array}
\]
and shows by a repeated application of Lemma \ref{lem:basic_decomposition_flat}
that the complete system is flat if and only if the subsystem
\begin{equation}
\begin{array}{ccc}
\bar{x}_{rest}^{+} & = & f_{rest}(\bar{x}_{rest},\bar{x}_{\bar{k}},\eta_{\bar{k}-1})\end{array}\label{eq:subsys_f_rest}
\end{equation}
with the inputs $(\bar{x}_{\bar{k}},\eta_{\bar{k}-1})$ is flat.

In the case $\dim(\Delta_{\bar{k}})=n$, because of $\dim(\bar{x}_{rest})=0$
the subsystem (\ref{eq:subsys_f_rest}) is an empty system with inputs
$(\bar{x}_{\bar{k}},\eta_{\bar{k}-1})$. Therefore, the complete system
is flat, and
\begin{equation}
y=(y_{\bar{k}},y_{\bar{k}-1},\ldots,y_{1})\label{eq:constructed_flat_output}
\end{equation}
with $y_{\bar{k}}=(\bar{x}_{\bar{k}},\eta_{\bar{k}-1})$ is a flat
output. The flat output (\ref{eq:constructed_flat_output}) consists
of the inputs of the (empty) system (\ref{eq:subsys_f_rest}), and
the redundant inputs $(y_{\bar{k}-1},\ldots,y_{1})$ of the subsystems
that have been eliminated during the repeated decompositions (cf.
Lemma \ref{lem:flatness_redundant_inputs}). The flat output in original
coordinates can be obtained by applying the inverse coordinate transformations.

For the case $\dim(\Delta_{\bar{k}})<n$, we show by contradiction
that the subsystem (\ref{eq:subsys_f_rest}) with $\dim(\bar{x}_{rest})>0$
cannot be flat. First, we eliminate all redundant inputs by an input
transformation
\[
(\zeta_{\bar{k}},y_{\bar{k}})=\Psi_{\bar{k}}(\bar{x}_{rest},\bar{x}_{\bar{k}},\eta_{\bar{k}-1})
\]
with inverse $(\bar{x}_{\bar{k}},\eta_{\bar{k}-1})=\hat{\Psi}_{\bar{k}}(\bar{x}_{rest},\zeta_{\bar{k}},y_{\bar{k}})$.
If the resulting system
\begin{equation}
\begin{array}{ccc}
\bar{x}_{rest}^{+} & = & f_{rest}(\bar{x}_{rest},\zeta_{\bar{k}})\end{array}\label{eq:subsys_f_rest_elim}
\end{equation}
would be flat, then according to the necessary condition of Theorem
\ref{thm:necessary_condition_flat} there would exist a nontrivial
vector field $v^{l}(\bar{x}_{rest},\zeta_{\bar{k}})\partial_{\zeta_{\bar{k}}^{l}}$
which is projectable with respect to the subsystem (\ref{eq:subsys_f_rest_elim}).
With respect to the complete system, such a vector field would be
contained in the largest projectable subdistribution $D_{\bar{k}}\subset E_{\bar{k}}$,
and accordingly the dimension of $\Delta_{\bar{k}+1}=f_{*}(D_{\bar{k}})$
would be larger than the dimension of $\Delta_{\bar{k}}=f_{*}(D_{\bar{k}-1})$.
However, because of $\dim(\Delta_{\bar{k}+1})=\dim(\Delta_{\bar{k}})$
such a vector field does not exist.
\begin{rem}
The Frobenius theorem, which is used for straightening out the distributions
(\ref{eq:chain_D}) and (\ref{eq:chain_Delta}), guarantees only local
results. Thus, it should be noted that if the distribution $D_{0}$
is considered locally around some point $(\bar{x},\bar{u})\in\mathcal{X}\times\mathcal{U}$,
then the distributions $\Delta_{1}=f_{*}(D_{0})$ and $D_{1}\subset E_{1}=\pi_{*}^{-1}(\Delta_{1})$
are defined locally around the points $f(\bar{x},\bar{u})\in\mathcal{X}^{+}$
and $(f(\bar{x},\bar{u}),\bar{u})\in\mathcal{X}\times\mathcal{U}$,
respectively. Likewise, $\Delta_{2}$ and $D_{2}$ are defined around
$f(f(\bar{x},\bar{u}),\bar{u})\in\mathcal{X}^{+}$ and $(f(f(\bar{x},\bar{u}),\bar{u}),\bar{u})\in\mathcal{X}\times\mathcal{U}$
(and so on). However, considering only a sufficiently small neighborhood
of an equilibrium $(x_{0},u_{0})$ ensures by $x_{0}=f(x_{0},u_{0})$
and the continuity of $f$ that the points $\bar{x},f(\bar{x},\bar{u}),f(f(\bar{x},\bar{u}),\bar{u}),\ldots$
are sufficiently close, such that the regions of validity of all coordinate
transformations overlap. In practice, the regions of validity of the
coordinate transformations -- and consequently the region of validity
of the proof of Theorem \ref{thm:necessary_and_sufficient_conditions}
-- can of course be quite large.
\end{rem}

\section{\label{sec:Example}Example}

In \cite{KolarSchoberlDiwold:2019}, we have already shown that the
system
\begin{equation}
\begin{array}{lcl}
x^{1,+} & = & \tfrac{x^{2}+x^{3}+3x^{4}}{u^{1}+2u^{2}+1}\\
x^{2,+} & = & x^{1}(x^{3}+1)(u^{1}+2u^{2}-3)+x^{4}-3u^{2}\\
x^{3,+} & = & u^{1}+2u^{2}\\
x^{4,+} & = & x^{1}(x^{3}+1)+u^{2}
\end{array}\label{eq:example_sys}
\end{equation}
is flat around the equilibrium $(x_{0},u_{0})=(0,0)$ by repeated
transformations into subsystems and endogenous dynamic feedbacks.
In the following, we prove its flatness again by
simply computing the sequence of distributions (\ref{eq:chain_Delta})
and applying Theorem \ref{thm:necessary_and_sufficient_conditions},
i.e., without actually performing any decompositions like in \cite{KolarSchoberlDiwold:2019}.
In the first step of Algorithm 1, we have to calculate the largest
projectable subdistribution of the distribution $E_{0}=\mathrm{span}\{\partial_{u^{1}},\partial_{u^{2}}\}$.
For this purpose, we introduce adapted coordinates (\ref{eq:adapted_coordinates})
on $\mathcal{X}\times\mathcal{U}$. After the transformation
\begin{equation}
\begin{array}{cclcccl}
\theta^{1} & = & f^{1}(x,u) & \quad & \xi^{1} & = & x^{1}\\
\theta^{2} & = & f^{2}(x,u) &  & \xi^{2} & = & x^{3}\,,\\
\theta^{3} & = & f^{3}(x,u)\\
\theta^{4} & = & f^{4}(x,u)
\end{array}\label{eq:example_adapted_coordinates}
\end{equation}
the vector fields $\partial_{u^{1}}$ and $\partial_{u^{2}}$ are
given by
\[
-\tfrac{\theta^{1}}{\theta^{3}+1}\partial_{\theta^{1}}+\xi^{1}(\xi^{2}+1)\partial_{\theta^{2}}+\partial_{\theta^{3}}
\]
and
\[
-2\tfrac{\theta^{1}}{\theta^{3}+1}\partial_{\theta^{1}}+(2\xi^{1}(\xi^{2}+1)-3)\partial_{\theta^{2}}+2\partial_{\theta^{3}}+\partial_{\theta^{4}}\,.
\]
Because of the presence of the fibre coordinates $\xi^{1}$ and $\xi^{2}$,
neither $\partial_{u^{1}}$ nor $\partial_{u^{2}}$ itself is projectable.
However, the linear combination $-2\partial_{u^{1}}+\partial_{u^{2}}$
reads in adapted coordinates as $-3\partial_{\theta^{2}}+\partial_{\theta^{4}}$,
and is therefore a projectable vector field. Thus, the largest projectable
subdistribution is given by
\[
D_{0}=\mathrm{span}\{-2\partial_{u^{1}}+\partial_{u^{2}}\}\,.
\]
The pushforward $f_{*}(D_{0})$ is the involutive distribution
\[
\Delta_{1}=\mathrm{span}\{-3\partial_{x^{2,+}}+\partial_{x^{4,+}}\}
\]
on $\mathcal{X}^{+}$ with $\dim(\Delta_{1})=\dim(D_{0})=1$. In the
second step, we have to determine the largest projectable subdistribution
of
\[
E_{1}=\mathrm{span}\{-3\partial_{x^{2}}+\partial_{x^{4}},\partial_{u^{1}},\partial_{u^{2}}\}\,.
\]
In adapted coordinates (\ref{eq:example_adapted_coordinates}), it
can be verified that the complete distribution is projectable, i.e.,
$D_{1}=E_{1}$. The pushforward $f_{*}(D_{1})$ is the involutive
distribution
\[
\begin{array}{cl}
\Delta_{2}=\mathrm{span}\{\hspace{-0.3cm} & -3\partial_{x^{2,+}}+\partial_{x^{4,+}},\,\tfrac{x^{1,+}}{x^{3,+}+1}\partial_{x^{1,+}}-\partial_{x^{3,+}},\\
 & \tfrac{2x^{1,+}}{x^{3,+}+1}\partial_{x^{1,+}}-2\partial_{x^{3,+}}-\partial_{x^{4,+}}\}
\end{array}
\]
with $\dim(\Delta_{2})=\dim(D_{1})=3$. In the third step, we have
to find the largest projectable subdistribution of
\[
\begin{array}{cl}
E_{2}=\mathrm{span}\{\hspace{-0.3cm} & -3\partial_{x^{2}}+\partial_{x^{4}},\,\tfrac{x^{1}}{x^{3}+1}\partial_{x^{1}}-\partial_{x^{3}},\\
 & \tfrac{2x^{1}}{x^{3}+1}\partial_{x^{1}}-2\partial_{x^{3}}-\partial_{x^{4}},\,\partial_{u^{1}},\,\partial_{u^{2}}\}\,.
\end{array}
\]
In adapted coordinates (\ref{eq:example_adapted_coordinates}), it
can be verified that again the complete distribution is projectable,
i.e., $D_{2}=E_{2}$. The pushforward $f_{*}(D_{2})$ is the involutive
distribution
\[
\Delta_{3}=\mathrm{span}\{\partial_{x^{1,+}},\partial_{x^{2,+}},\partial_{x^{3,+}},\partial_{x^{4,+}}\}\,.
\]
Here we have $\dim(\Delta_{3})<\dim(D_{2})=5$. However, because of
$\dim(\Delta_{3})=n=4$ we can stop, and according to Theorem \ref{thm:necessary_and_sufficient_conditions}
the system (\ref{eq:example_sys}) is flat. It is important to emphasize
that all these computations require only the solution of algebraic
equations, and can be performed efficiently with a computer algebra
program.

Now let us calculate a flat output. For this purpose, in a first step
we straighten out the sequence (\ref{eq:chain_Delta}) by a state
transformation of the form (\ref{eq:state_transformation_Delta})
with $\bar{x}_{1}=\bar{x}_{1}^{1}$, $\bar{x}_{2}=(\bar{x}_{2}^{1},\bar{x}_{2}^{2})$,
$\bar{x}_{3}=\bar{x}_{3}^{1}$ and $\dim(\bar{x}_{rest})=0$. With
\[
\begin{array}{cclcclccc}
\bar{x}_{1}^{1} & = & x^{4}\,, & \bar{x}_{2}^{1} & = & x^{2}+3x^{4}\,, & \bar{x}_{3}^{1} & = & x^{1}(x^{3}+1)\\
 &  &  & \bar{x}_{2}^{2} & = & x^{3}
\end{array}
\]
we get
\[
\begin{array}{ccl}
\Delta_{1} & = & \mathrm{span}\{\partial_{\bar{x}_{1}^{1,+}}\}\\
\Delta_{2} & = & \mathrm{span}\{\partial_{\bar{x}_{1}^{1,+}},\partial_{\bar{x}_{2}^{1,+}},\partial_{\bar{x}_{2}^{2,+}}\}\\
\Delta_{3} & = & \mathrm{span}\{\partial_{\bar{x}_{1}^{1,+}},\partial_{\bar{x}_{2}^{1,+}},\partial_{\bar{x}_{2}^{2,+}},\partial_{\bar{x}_{3}^{1,+}}\}\,,
\end{array}
\]
and the transformed system (\ref{eq:sys_straightened_out_Delta})
reads
\begin{equation}
\begin{array}{ccl}
\bar{x}_{3}^{1,+} & = & \bar{x}_{2}^{1}+\bar{x}_{2}^{2}\\
\bar{x}_{2}^{1,+} & = & \bar{x}_{1}^{1}+\bar{x}_{3}^{1}(u^{1}+2u^{2})\\
\bar{x}_{2}^{2,+} & = & u^{1}+2u^{2}\\
\bar{x}_{1}^{1,+} & = & \bar{x}_{3}^{1}+u^{2}\,.
\end{array}\label{eq:example_sys_straightened_out_delta}
\end{equation}
In a second step, we combine the coordinate transformations on $\mathcal{X}\times\mathcal{U}$
that are constructed in the proof of Theorem \ref{thm:necessary_and_sufficient_conditions},
and obtain a coordinate transformation
\begin{equation}
\begin{array}{cclcccl}
\bar{x}_{3}^{1} & = & y_{3}^{1} & \quad & u^{1} & = & \hat{z}_{1}^{2}-2\hat{z}_{0}^{1}\\
\bar{x}_{2}^{1} & = & y_{2}^{1} &  & u^{2} & = & \hat{z}_{0}^{1}\\
\bar{x}_{2}^{2} & = & \hat{z}_{2}^{1}-y_{2}^{1}\\
\bar{x}_{1}^{1} & = & \hat{z}_{1}^{1}
\end{array}\label{eq:example_coordinate_transformation_D_combined}
\end{equation}
that straightens out the sequence (\ref{eq:chain_D}) according to
\[
\begin{array}{ccl}
D_{0} & = & \mathrm{span}\{\partial_{\hat{z}_{0}^{1}}\}\\
D_{1} & = & \mathrm{span}\{\partial_{\hat{z}_{0}^{1}},\partial_{\hat{z}_{1}^{1}},\partial_{\hat{z}_{1}^{2}}\}\\
D_{2} & = & \mathrm{span}\{\partial_{\hat{z}_{0}^{1}},\partial_{\hat{z}_{1}^{1}},\partial_{\hat{z}_{1}^{2}},\partial_{y_{2}^{1}},\partial_{\hat{z}_{2}^{1}}\}\,.
\end{array}
\]
Applying the transformation (\ref{eq:example_coordinate_transformation_D_combined})
and its shifted version to the right-hand side and the left-hand side
of the system equations (\ref{eq:example_sys_straightened_out_delta})
results in a structurally flat implicit\footnote{The coordinate transformations on $\mathcal{X}\times\mathcal{U}$
constructed in the proof of Theorem \ref{thm:necessary_and_sufficient_conditions}
can be interpreted as input transformations for the corresponding
subsystems. However, with regard to the complete system, they do not
necessarily preserve the separation into state- and input variables.
Thus, the transformation can yield in general an implicit system representation.
For the system (\ref{eq:example_triangular_form}), however, by rearranging
the equations of the block $\Xi_{2}$ also an explicit triangular
form would be possible.} triangular form
\begin{equation}
\begin{array}{cl}
\Xi_{3}: & \begin{array}{c}
y_{3}^{1,+}-\hat{z}_{2}^{1}=0\end{array}\\
\Xi_{2}: & \begin{array}{l}
y_{2}^{1,+}-y_{3}^{1}\hat{z}_{1}^{2}-\hat{z}_{1}^{1}=0\\
\hat{z}_{2}^{1,+}-y_{2}^{1,+}-\hat{z}_{1}^{2}=0
\end{array}\\
\Xi_{1}: & \begin{array}{c}
\hat{z}_{1}^{1,+}-y_{3}^{1}-\hat{z}_{0}^{1}=0\end{array}
\end{array}\label{eq:example_triangular_form}
\end{equation}
as it is discussed in \cite{KolarSchoberlSchlacher:2016-2} (see \cite{SchoberlSchlacher:2014}
for a continuous-time counterpart). This triangular representation
allows to read off a flat output $y=(y_{3}^{1},y_{2}^{1})$ and to
systematically determine the parameterization of the other system
variables by evaluating the equations (\ref{eq:example_triangular_form})
from top to bottom. From the topmost block $\Xi_{3}$ we get the parameterization
of the variable $\hat{z}_{2}^{1}$, from $\Xi_{2}$ we get $\hat{z}_{1}^{1}$
and $\hat{z}_{1}^{2}$, and from $\Xi_{1}$ we finally get $\hat{z}_{0}^{1}$.
In original coordinates, the flat output is given by $y=(x^{1}(x^{3}+1),x^{2}+3x^{4})$.